\documentclass[12 pt]{article}%
\usepackage{amsmath, amsfonts, amsthm, color,latexsym}
\usepackage{amsmath}
\usepackage{amsfonts}
\usepackage{amssymb}
\usepackage{color, soul}
\usepackage{enumerate}
\providecommand{\U}[1]{\protect\rule{.1in}{.1in}}
\newtheorem{theorem}{Theorem}[section]
\newtheorem{proposition}[theorem]{Proposition}
\newtheorem{corollary}[theorem]{Corollary}

\newtheorem{examples}[theorem]{Examples}
\newtheorem{remark}[theorem]{Remark}

\newtheorem{lemma}[theorem]{Lemma}
\newtheorem{final remark}[theorem]{Final Remark}
\newtheorem{definition}[theorem]{Definition}
\allowdisplaybreaks[4]

\begin{document}

\title{\sc Transformations of sequence spaces by multipolinomials}
\author{Joilson Ribeiro,
	~ Fabr\'icio Santos\thanks{Supported by CAPES Doctoral scholarship.	
	\thinspace \hfill\newline\indent2010 Mathematics Subject
Classification: 46B45, 47L22, 47B10.\newline\indent Key words: Banach sequence spaces, ideals of multilinear operators, multiple summing operators.}}
\date{}
\maketitle

\begin{abstract}
In this paper we introduce a new approach to the concept of multipolynomials and generalize several results of the homogeneous polynomials and symmetric multilinear applications. We also present an abstract approach to the concept of absolutely summing multipolynomials.
\end{abstract}

\section{Introduction and background}

The concept of multipolynomial was introduced by T. Velanga in \cite{V17, V17MP}, as an attempt to unify the successful theories of homogeneous polynomials and multilinear applications between Banach spaces. Among others advantages, the theory of multipolynomials contemplates, as a particular case, the multilinear operators and homogeneous polinomials theory. 

In \cite{V17MP} it has been shown that each multipolynomial is, in particular, a homogeneous polynomial and, consequently, it can be associated to a single symmetric multilinear application, that is, each $(n_1,\dots,n_m)$-homogeneous polynomial $P\colon E_1\times\cdots\times E_m\rightarrow F$ can be associated to the symmetric application $\check{P}\colon\left(E_1\times\cdots\times E_m\right)^{\left(n_1+\cdots+n_m\right)}\rightarrow F$. Such application, although obtained in a natural way, presents many difficulties in practical use, due to the large size of its domain.


In an attempt to mitigate such difficulty, we introduce a new approach for multipolynomials where each of them is associated with a multilinear application with domain $E_1^{n_1}\times\cdots\times E_m^{n_m}$ acting on $F$. One of our first goals in this work is to investigate, under what conditions, the abovementioned application $T\colon E_1^{n_1}\times\cdots\times E_m^{n_m}\rightarrow F$ is unique. For this it was necessary to generalize the concept of symmetry and the famous Polarization Formula.



This approach allows us to explore the class of absolutely summing multipolynomials, which in the case of multilinear applications and homogeneous polynomials have been studied by several authors, such as: the absolutely summing \cite{BBDP07, M03}, the Cohen strongly summing \cite{C14}, the $(s, q, p_1,\dots, p_n)$-mixing summing \cite{BPSS15, M04}, the mid summing \cite{BCS17, KS14, RS17} and the almost summing \cite{PR12}. As such constructs followed a similar script, there arose the natural concern of constructing an abstract approach for such multilinear operators and the homogeneous polynomials \cite{BC17, RS17, S13}. In this spirit, we will construct in Section \ref{AbsolutelyMultipolynomials} an abstract approach to the absolutely summing multipolynomials.


From now on, $E$, $E_1,\dots,E_m$, $F$, $G_1,\dots,G_m$, $H$ shall denote Banach spaces over $\mathbb{K}=\mathbb{R}$ or $\mathbb{C}$. By ${\cal L}(E_1, \ldots, E_n;F)$ we denote the Banach space of continuous $n$-linear operators from $E_1 \times \cdots \times E_n$ to $F$ endowed with the usual uniform norm $\|\cdot\|$.

%

The next two definitions were introduced in \cite{BC17}.

\begin{definition}{}
A class of vector-valued sequences $\gamma_s$, is a rule that assigns to each Banach space $E$ a Banach space $\gamma_s(E)$  of $E$-valued sequences, that is
$\gamma_s(E)$ is a vector subspace of $E^{\mathbb{N}}$ with the coordinatewise operations, such that:

\begin{equation*}
c_{00}(E) \subset \gamma_s(E) \overset{1}{\hookrightarrow} \ell_{\infty}(E) \quad \text{and} \quad \|e_j\|_{\gamma_s(\mathbb{K})} = 1, \quad \forall j=1,\dots, n.
\end{equation*}
where $e_j$ is the vector with $1$ in the $j$-th coordinate and zero in the other coordinates, and
the symbol $E \overset{1}{\hookrightarrow} F$ means that $E $ is a linear subspace of $F$ and $\|x\|_F \le \|x\|_E$ .
\end{definition}

\begin{definition}
\begin{description}
\item $(i)$ A sequence class $\gamma_s$ is \textit{finitely determined} if for every $(x_j)_{j=1}^{\infty} \in E^{\mathbb{K}}$,
\begin{equation*}
(x_j)_{j=1}^{\infty} \in \gamma_s(E) \iff \sup_k\left\|\left(x_j \right)_{j=1}^{k} \right\|_{\gamma_s(E)} < \infty.
\end{equation*}
In this case, 
\begin{equation*}
\left\|\left(x_j \right)_{j=1}^{\infty} \right\|_{\gamma_s(E)} := \sup_k\left\|\left(x_j \right)_{j=1}^{k} \right\|_{\gamma_s(E)}.
\end{equation*}

\item $(ii)$ A sequence class $\gamma_s$ is said to be linearly stable if for every $u \in \mathcal{L}(E; F)$
\begin{equation*}
\left(u\left(x_j \right)\right)_{j=1}^{\infty} \in \gamma_s(F)
\end{equation*}
holds whenever $\left(x_j \right)_{j=1}^{\infty} \in \gamma_s(E)$ and $\|\hat{u} : \gamma_s(E) \rightarrow \gamma_s(F)\| = \|u\|$.

\item $(iii)$ Given sequence classes $\gamma_{s_1},\dots, \gamma_{s_n}, \gamma_s$. We say that $\gamma_{s_1}(\mathbb{K})\cdots \gamma_{s_n}(\mathbb{K}) \overset{1}{\hookrightarrow} \gamma_s(\mathbb{K})$ if $\left(\lambda_j^{(1)}\cdots \lambda_j^{(n)} \right)_{j=1}^{\infty} \in \gamma_s(\mathbb{K})$ and
\begin{equation*}
\left\|\left(\lambda_j^{(1)}\cdots \lambda_j^{(n)} \right)_{j=1}^{\infty} \right\|_{\gamma_s(\mathbb{K})} \le \prod_{i=1}^{\infty}\left\|\left(\lambda_j^{(i)} \right)_{j=1}^{\infty} \right\|_{\gamma_{s_i}(\mathbb{K})}
\end{equation*}
whenever $\left(\lambda_j^{(i)} \right)_{j=1}^{\infty} \in \gamma_{s_i}(\mathbb{K}), i=1,\dots, n$.
\end{description}
\end{definition}

\section{Multipolynomials $($New approach$)$}

We will present in this section a new approach to $(n_1,\dots, n_m)$-homogeneous polynomials. This approach associates each $(n_1,\dots, n_m)$-homogeneous polynomial with an $(n_1+\cdots + n_m)$-linear application. We further define a new concept of symmetry and generalize the classical Polarization Formula.

\begin{definition}\rm Let $m\in\mathbb{N}$ and $(n_1,\ldots,n_m)\in\mathbb{N}^m$. A mapping $P\colon E_1\times\cdots\times E_m\longrightarrow F$ is said to be an \textit{$(n_1,\ldots,n_m)$-homogenous polynomial} if, for each $i=1,\ldots,m$, the mapping $$P(x_1,\ldots,x_{i-1},\cdot,x_{i+1},\ldots,x_m)\colon E_i\longrightarrow F$$ is an $n_i$-homogenous polynomial for all fixed $x_1\in E_1,\ldots,x_{i-1}\in E_{i-1},x_{i+1}\in E_{i+1},\ldots,x_m\in E_m$. The mapping $P$ is a {\it multipolynomial} if it is an $(n_1,\ldots,n_m)$-homogenous polynomial for some $(n_1,\ldots,n_m)\in\mathbb{N}^m$. The space of all continuous $(n_1,\ldots,n_m)$-homogenous polynomials from $E_1\times\cdots\times E_m$ to $F$ is denoted by ${\cal MP}\left(^{n_1}E_1,\ldots,^{n_m}E_m;F\right)$. When $F=\mathbb{K}$, we simply write ${\cal MP}\left(^{n_1}E_1,\ldots,^{n_m}E_m\right)$. Defining
$$\|P\| = \sup\left\{\|P(x_1, \ldots, x_m)\| : x_i \in B_{E_1}, i = 1,\ldots, n \right\}, $$
${\cal MP}\left(^{n_1}E_1,\ldots,^{n_m}E_m;F\right)$ becomes a Banach space. \end{definition}

\begin{theorem}\label{CMP}
$P$ is an $(n_1,\dots, n_m)$-homogeneous polynomial if, and only if, there is $T \in L(^{n_1}E_1,\dots, ^{n_m}E_m)$ such that
\begin{equation*}
P(x_1,\dots, x_m) = T(x_1,\overset{n_1}{\dots}, x_1, x_2,\overset{n_2}{\dots}, x_2,\dots, x_{m},\overset{n_m}{\dots}, x_{m}).
\end{equation*}
\end{theorem}

\begin{proof}
''$\Leftarrow$'' is immediate. 

''$\Rightarrow$'' In \cite[Theorem $2.1$]{V17MP} it was shown that any $(n_1,\dots, n_m)$-homogeneous polynomials $P \colon E_1\times \cdots \times E_m \rightarrow F$ can be written in the form
\begin{align*}
&P(x_1,\dots, x_m) =\\
&\sum_{j_1^{(1)},\dots, j_{n_1}^{(1)} \in J_1}\cdots \sum_{j_1^{(m)},\dots, j_{n_m}^{(m)} \in J_m}b_{j_1^{(1)},\dots, j_{n_1}^{(1)},\dots, j_1^{(m)},\dots j_{n_m}^{(m)}}\xi_{j_1^{(1)}}^{(1)}\cdots \xi_{j_{n_1}^{(1)}}^{(1)}(x_1)\cdots \xi_{j_1^{(m)}}^{(m)}\cdots \xi_{j_{n_m}^{(m)}}^{(m)}(x_m).
\end{align*}
where
\begin{equation*}
b_{j_1^{(1)},\dots, j_{n_1}^{(1)},\dots, j_1^{(m)},\dots j_{n_m}^{(m)}} = \frac{1}{2^{n_1+\dots +n_m}n_1!\cdots n_m!}\sum_{\epsilon_i^{(j)} = \pm 1}\left(\prod_{i=1}^{m}\epsilon_1^{(i)}\cdots \epsilon_{n_m}^{(i)}\right)P\left(\sum_{k=1}^{n_1}\epsilon_k^{(1)}e_k^{(1)},\dots, \sum_{k=1}^{n_m}\epsilon_k^{(m)}e_k^{(m)} \right).
\end{equation*}

Now define
\begin{equation*}
T\colon E_1\times \overset{n_1}{\cdots}\times E_1 \times  \cdots \times E_m\overset{n_m}{\cdots}\times E_m \rightarrow F
\end{equation*}
given by,
\begin{align*}
&T\left(x_1^{(1)},\dots, x_{n_1}^{(1)},\dots, x_1^{(m)},\dots, x_{n_m}^{(m)} \right) =\\
&\sum_{j_1^{(1)},\dots, j_{n_1}^{(1)} \in J_1}\cdots \sum_{j_1^{(m)},\dots, j_{n_m}^{(m)} \in J_m}b_{j_1^{(1)},\dots, j_{n_1}^{(1)},\dots, j_1^{(m)},\dots j_{n_m}^{(m)}}\xi_{j_1}^{(1)}\left(x_1^{(1)} \right)\cdots \xi_{j_{n_1}}^{(1)}\left(x_{n_1}^{(1)}\right)\cdots \xi_{j_1}^{(m)}\left(x_1^{(m)} \right)\cdots \xi_{j_{n_m}}^{(m)}\left(x_{n_m}^{(m)}\right).
\end{align*}
It is easy to see that $T$ is an $(n_1+\cdots + n_m)$-linear application and
\begin{equation*}
P(x_1,\dots, x_m) = T\left(x_1,\overset{n_1}{\dots},x_1,\dots, x_m,\overset{n_m}{\dots}, x_m \right).
\end{equation*}
\end{proof}

The next definition presents the concept of $(n_1,\dots, n_m)$-symmetric multilinear applications, which will be sufficient to ensure that each multi-polynomial can be associated with a unique $(n_1,\dots, n_m)$-symmetric multilinear application.

\begin{definition}
Let $T \in \mathcal{L}(^{n_1}E_1,\dots, ^{n_m}E_{n_m})$. $T$ is said $(n_1,\dots, n_m)$-symmetric when
\begin{align*}
T\left(x_{\sigma_1(1)}^{(1)},\dots, x_{\sigma_1(n_1)}^{(1)},\dots, x_{\sigma_m(1)}^{(m)},\dots, x_{\sigma_m(n_m)}^{(m)} \right) = T\left(x_1^{(1)},\dots, x_{n_1}^{(1)},\dots, x_1^{(m)},\dots, x_{n_m}^{(m)} \right),
\end{align*}
for every $\sigma_i \in S_{n_i}$, $i=1,\dots, m$.
\end{definition}

\begin{examples}\label{example_N_simmetrical}
\begin{description}
\item $(1)$ Let $E_1 = \mathbb{R}$ and $E_2 = \mathbb{C}$. Then the $4$-linear application $T\colon E_1 \times E_1 \times E_2 \times E_2 \rightarrow \mathbb{C}$, given by
\begin{equation*}
T(x, y, \alpha, \beta) = xy\alpha \beta
\end{equation*}
is a $(2, 2)$-symmetric application, but it is not symmetric, because $E_1 \neq E_2$.

\item $(2)$ Let $E_1 = E_2 = \mathbb{R}^2$. Then, the $4$-linear application $T\colon E_1\times E_1 \times E_2 \times E_2 \rightarrow \mathbb{R}$ given by
\begin{equation*}
T\left((x, y), (z, w), (a, b), (c, d) \right) = xzbd
\end{equation*}
is a $(2, 2)$-symmetric application, but it is not symmetric, even if $E_1 = E_2$.
\end{description}
\end{examples}

The linear space of all continuous $(n_1,\dots, n_m)$-symmetric applications is denoted by $\mathcal{L}_s^{(n_1,\dots,n_m)}(^{n_1}E_1,\dots, ^{n_m}E_m; F)$.

\begin{lemma}
If $dim(E)>1$, then $\mathcal{L}_s(^{n_1}E,\dots, ^{n_m}E; F) \subsetneq \mathcal{L}_s^{(n_1,\dots, n_m)}(^{n_1}E,\dots, ^{n_m}E; F)$.
\end{lemma}

\begin{proof}
The inclusion is immediate and item $(2)$ of the Example $\ref{example_N_simmetrical}$ shows that the inclusion is proper.  
\end{proof}

\begin{remark}
	If $dim(E)=1$, then $\mathcal{L}_s(^{n_1}E,\dots, ^{n_m}E; F) = \mathcal{L}_s^{(n_1,\dots, n_m)}(^{n_1}E,\dots, ^{n_m}E; F)$.
\end{remark}

The next result shows that every $(n_1+\cdots +n_m)$-linear application is $(n_1,\dots, n_m)$-symmetrizable.

\begin{lemma}
Let $T \in  \mathcal{L}(^{n_1}E_1,\dots, ^{n_m}E_m; F)$. Then, the application 
\begin{align*}
&T_s^{(n_1,\dots, n_m)}\left(x_1^{(1)},\dots, x_{n_1}^{(1)},\dots, x_{1}^{(m)},\dots, x_{n_m}^{(m)}\right)\\
&= \frac{1}{n_1!\cdots n_m!}\sum_{\begin{array}{cc}
\sigma_i \in S_{n_i}\\
i=1,\dots, m
\end{array}} T_{\sigma_1.\dots, \sigma_m}\left(x_1^{(1)},\dots, x_{n_1}^{(1)},\dots, x_{1}^{(m)},\dots, x_{n_m}^{(m)} \right),
\end{align*}
where
\begin{align*}
T_{\sigma_1.\dots, \sigma_m}\left(x_1^{(1)},\dots, x_{n_1}^{(1)},\dots, x_{1}^{(m)},\dots, x_{n_m}^{(m)} \right) = T\left(x_{\sigma_1(1)}^{(1)},\dots, x_{\sigma_1(n_1)}^{(1)},\dots, x_{\sigma_m(1)}^{(1)},\dots, x_{\sigma_m(n_m)}^{(m)} \right).
\end{align*}
is $(n_1,\dots, n_m)$-symmetric. This application is called $(n_1,\dots, n_m)$-simmetrization of $T$.
\end{lemma}


The next result, which is a consequence of the Polarization Formula for symmetric multilinear applications, will give us a formula that will be extremely important in proving the results that follow.

\begin{theorem}{$($\textbf{$(n_1,\dots, n_m)$-Polarization Formula}$)$}
Let $E_1,\dots, E_m, F$ be Banach space. If $T \in \mathcal{L}_s^{(n_1,\dots, n_m)}(^{n_1}E_1,\dots, ^{n_m}E_m; F)$, then
\begin{align*}
&T\left(x_1^{(1)},\dots,x_{n_1}^{(1)}, x_{1}^{(2)},\dots, x_{n_2}^{(2)},\dots,  x_{1}^{(m)},\dots, x_{n_m}^{(m)}\right)\\
=&\mathcal{A}_{n_1,\dots,n_m}\prod_{k=1}^{m}\left(\sum_{\epsilon_i^{(k)} = \pm 1}\epsilon_1^{(k)}\cdots \epsilon_{n_k}^{(k)}\right) T\left(\left(x_0^{(1)} + \sum_{i=1}^{n_1}\epsilon_i^{(1)} x_i^{(1)}\right)^{n_1},\dots, \left(x_0^{(m)} + \sum_{i=1}^{n_m}\epsilon_i^{(m)} x_i^{(m)}\right)^{n_m} \right)
\end{align*}
for every $x_0^{(j)}, x_1^{(j)},\dots, x_{n_j}^{(j)} \in E_j$, $j=1,\dots,m$, where $\mathcal{A}_{n_1,\dots,n_m}=\displaystyle\frac{1}{2^{n_1+\cdots + n_m}n_1! \cdots  n_m!}$.
\end{theorem}

An immediate consequence of the $(n_1,\dots, n_m)$-Polarization Formula, is the next result.

\begin{corollary}\label{PI}
\begin{description}
\item $(a)$ Let $E, F$ be Banach space and $T, R \in  \mathcal{L}_s^{(n_1,\dots, n_m)}(^{n_1}E_1,\dots, ^{n_m}E_m; F)$. If $T(x_1,\overset{n_1}{\dots}, x_{1},\dots, x_{m},\overset{n_m}{\dots}, x_m) = R(x_1,\overset{n_1}{\dots}, x_{1},\dots, x_{m},\overset{n_m}{\dots}, x_m)$, then $$T(x_1^{(1)},\dots, x_{n_1}^{(1)},\dots, x_{1}^{(m)},\dots, x_{n_m}^{(m)}) = R(x_1^{(1)},\dots, x_{n_1}^{(1)},\dots, x_{1}^{(m)},\dots, x_{n_m}^{(m)}),$$ for every $ x_1^{(i)},\dots, x_{n_i}^{(i)} \in E_i$, $i=1,\dots,m$.

\item $(b)$ Let $T \in  \mathcal{L}(^{n_1}E_1,\dots, ^{n_m}E_m; F)$. Then
\begin{equation*}
T(x_1,\overset{n_1}{\dots}, x_{1},\dots, x_{m},\overset{n_m}{\dots}, x_{m}) = T_s^{(n_1,\dots, n_m)}(x_1,\overset{n_1}{\dots}, x_{1},\dots, x_{m},\overset{n_m}{\dots}, x_{m}).
\end{equation*}
\end{description}

\end{corollary}

The next result states that for each $(n_1,\dots, n_m)$-homogeneous polynomial, there is a unique $(n_1,\dots, n_m)$-symmetric $(n_1+\dots+ n_m)$-linear application associated to it.

\begin{proposition}\label{UNIC}
Let $P \in \mathcal{MP}(^{n_1}E_1,\dots, ^{n_m}E_m; F)$. Then there is a unique $$\check{P} \in \mathcal{L}_s^{(n_1,\dots, n_m)}(^{n_1}E_1,\dots. ^{n_m}E_m; F)$$ such that
\begin{equation*}
P(x_1,\dots, x_m) = \check{P}(x_1,\overset{n_1}{\dots}, x_{1},\dots, x_{m},\overset{n_m}{\dots}, x_{m}).
\end{equation*}
\end{proposition}

\begin{proof}
Let $P \in \mathcal{MP}(^{n_1}E_1,\dots, ^{n_m}E_m; F)$. By Theorem \ref{CMP} and Corollary \ref{PI} item $(b)$, there is $T \in \mathcal{L}(^{n_1}E_1,\dots, ^{n_m}E_m; F)$ such that
\begin{align*}
P(x_1,\dots, x_m) &= T(x_1,\overset{n_1}{\dots}, x_{1},\dots, x_{m},\overset{n_m}{\dots}, x_{m})\\
&= T_s^{(n_1,\dots, n_m)}(x_1,\overset{n_1}{\dots}, x_{1},\dots, x_{m},\overset{n_m}{\dots}, x_{m}).
\end{align*}
This shows the existence. It remains to prove the uniqueness. Suppose that there exists $R \in \mathcal{L}_s^{(n_1,\dots, n_m)}(^{n_1}E_1,\dots. ^{n_m}E_m; F)$, such that
\begin{equation*}
P(x_1,\dots, x_m) = R(x_1,\overset{n_1}{\dots}, x_{1},\dots, x_{m},\overset{n_m}{\dots}, x_{m}).
\end{equation*}
Then
\begin{equation*}
\check{P}(x_1,\overset{n_1}{\dots}, x_{1},\dots, x_{m},\overset{n_m}{\dots}, x_{m}) = R(x_1,\overset{n_1}{\dots}, x_{1},\dots, x_{m},\overset{n_m}{\dots}, x_{m}),
\end{equation*}
and the result follows from Proposition \ref{PI} item $(a)$.
\end{proof}

It is known that each multilinear application can be associated with a homogeneous polynomial. From now on, we can associate, in an analogous way, each multilinear application $T \in \mathcal{L}(^{n_1}E_1,\dots, ^{n_m}E_m; F)$  to an $(n_1,\dots, n_m)$-homogeneous polynomial, which will be denoted by $\hat{T}$ and defined as follows: 
\begin{equation*}
 \hat{T}(x_1,\dots, x_m):=T(x_1,\overset{n_1}{\dots}, x_1, \dots, x_m, \overset{n_m}{\dots}, x_m).
\end{equation*}


\begin{proposition}
Let $E_1,\dots, E_m, F$ be Banach spaces. The application
\begin{equation*}
\Phi : \mathcal{L}_s^{(n_1,\dots, n_m)}(^{n_1}E_1,\dots, ^{n_m}E_m; F) \rightarrow \mathcal{MP}(^{n_1}E_1,\dots, ^{n_m}E_m; F)
\end{equation*}
defined by $\Phi(T) = \hat{T}$ is an isomorphism. Furthermore,
\begin{equation*}
\|\hat{T}\| \le \|T\| \le \frac{n_1^{n_1}\cdots n_m^{n_m}}{n_1!\cdots n_m!}\|\hat{T}\|.
\end{equation*}
\end{proposition}

\begin{proof}
It is not difficult to see that $\Phi$ is a isomorphism. 

Let $T \in \mathcal{L}_s^{(n_1,\dots, n_m)}(^{n_1}E_1,\dots, ^{n_m}E_m; F)$. Then
\begin{equation*}
\|\hat{T}(x_1,\dots, x_m)\| = \|T(x_1,\overset{n_1}{\dots}, x_1,\dots, x_m,\overset{n_m}{\dots}, x_m))\| \le \|T\| \|x_1\|^{n_1}\cdots \|x_m\|^{n_m}.
\end{equation*}
Thus, $\|\hat{T}\| \le \|T\|$. The other inequality is an immediate consequence of the $(n_1, \dots, n_m)$-Polarization Formula.

\end{proof}

\begin{corollary}
Let $P \in \mathcal{MP}(^{n_1}E_1.\dots, ^{n_m}E_m; F)$. So,
\begin{equation*}
\|P\| \le \|\check{P}\| \le \frac{n_1^{n_1}\cdots n_m^{n_m}}{n_1!\cdots n_m!}\|P\|.
\end{equation*}
\end{corollary}

In the next theorem, we will use the following notation	$$\mathcal{MP}^{(n_1,\dots, n_m)}:=\mathcal{MP}(^{n_1}(\cdot).\dots, ^{n_m}(\cdot);(\cdot)).$$

\begin{theorem}
Let $\mathcal{M}$ be an $(n_1+\cdots + n_m)$-linear applications ideal. Define
\begin{equation*}
\mathcal{MP}_{\mathcal{M}}:= \{P \in \mathcal{MP}^{(n_1,\dots, n_m)}\text{; } \check{P} \in \mathcal{M} \}.
\end{equation*}
Then, $\mathcal{MP}_{\mathcal{M}}$ is an $(n_1,\dots, n_m)$-homogeneous polynomials ideal. If $\left(\mathcal{M}, \|\cdot\|_{\mathcal{M}} \right)$ is a normed $(n_1+\cdots + n_m)$-linear applications ideal, then $\left(\mathcal{MP}_{\mathcal{M}}, \|\cdot\|_{\mathcal{MP}_{\mathcal{M}}} \right)$ is a normed  $(n_1,\dots, n_m)$-homogeneous polynomials ideal, where 
\begin{equation*}
\|P\|_{\mathcal{MP}_{\mathcal{M}}} := \|\check{P} \|_{\mathcal{M}}.
\end{equation*} 

\end{theorem}

\begin{proof}
It is not difficult to see that $\mathcal{MP}_{\mathcal{M}}(^{n_1}E_1,\dots, ^{n_m}E_m; F)$ is a linear subspace of $\mathcal{MP}(^{n_1}E_1,\dots, ^{n_m}E_m; F)$, and using item $(a)$ of Proposition \ref{PI} we conclude that the $(n_1,\dots, n_m)$-homogeneous polynomials of finite type belongs to it.

To prove the ideal property, let $P \in \mathcal{MP}_{\mathcal{M}}(^{n_1}E_1,\dots, ^{n_m}E_m; F)$, $t \in \mathcal{L}(F; H)$ and $u_j \in \mathcal{L}(G_j; E_j)$. Considering $x_0^{(1)},\dots, x_0^{(m)}$ equal to zero in the $(n_1,\dots, n_m)$-Polarization Formula, we can conclude that 
$(t\circ P \circ (u_1,\dots, u_m))^{\vee} = t\circ \check{P} \circ \left(u_1,\overset{n_1}{\dots}, u_1, \dots, u_m,\overset{n_m}{\dots}, u_m \right)$. Therefore, from $$t\circ \check{P} \circ \left(u_1,\overset{n_1}{\dots}, u_1, \dots, u_m,\overset{n_m}{\dots}, u_m \right) \in \mathcal{M}(^{n_1}G_1,\dots, ^{n_m}G_m; H),$$ it follows that
\begin{equation*}
(t\circ P \circ (u_1,\dots, u_m))^{\vee} \in \mathcal{M}(^{n_1}G_1,\dots, ^{n_m}G_m; H).
\end{equation*}
Thus, $\mathcal{MP}_{\mathcal{M}}$ is an $(n_1,\dots, n_m)$-homogeneous polynomials ideal. Now, if $\left(\mathcal{M}, \|\cdot\|_{\mathcal{M}} \right)$ is a normed $(n_1+\cdots + n_m)$-linear applications ideal, it is easy to see that $\|P\|_{\mathcal{MP}_{\mathcal{M}}} $ is a norm on $\mathcal{MP}_{\mathcal{M}}(^{n_1}E_1,\dots, ^{n_m}E_m; F)$, and defining
$Id_{\mathbb{K}}^{(n_1,\dots, n_m)}(\lambda_1,\dots, \lambda_m) = \lambda_1^{n_1}\cdots \lambda_m^{n_m}$,
it follows immediately from item $(a)$ of Proposition \ref{PI} that $$\|Id_{\mathbb{K}}^{(n_1,\dots, n_m)}\|_{\mathcal{P}_{\mathcal{M}}} = 1.$$



Finally, using the ideal property, it follows that
\begin{align*}
\left\|t\circ P \circ (u_1,\dots, u_m) \right\|_{\mathcal{MP}_{\mathcal{M}}}
&:= \left\|(t\circ P \circ (u_1,\dots, u_m))^{\vee} \right\|_{\mathcal{M}}\\
&\le \|t\| \left\|P \right\|_{\mathcal{MP}_{\mathcal{M}}}\|u_1\|^{n_1}\cdots\|u_m\|^{n_m}.
\end{align*}
\end{proof}

\section{Absolutely summing multipolynomials}\label{AbsolutelyMultipolynomials}

In this section, we will introduce the class of absolutely $\gamma_{s, s_1,\dots, s_m}$-summing $(n_1,\dots, n_m)$-homogeneous polynomials and present a norm that makes it a complete class.

\begin{definition}\label{ashmp}
We say that a continuous $(n_1,\dots, n_m)$-homogeneous polynomial $P : E_1 \times \cdots \times E_m \longrightarrow F$ is absolutely  $\gamma_{s, s_1,\dots, s_m}$-summing in $a = (a_1,\dots, a_m) \in E_1 \times \dots \times E_m$ if
\begin{equation*}
\left(P\left(a_1 + x_j^{(1)},\dots, a_m + x_j^{(m)} \right) - P(a_1,\dots, a_m) \right)_{j=1}^{\infty} \in \gamma_s(F),
\end{equation*}
whenever $\left(x_j^{(i)} \right)_{j=1}^{\infty} \in \gamma_{s_i}(E_i)$, $i=1,\dots, m$.
\end{definition}

The set of all  $(n_1,\dots, n_m)$-homogeneous polynomials which are absolutely $\gamma_{s, s_1,\dots, s_m}$-summing in $a  = (a_1,\dots, a_m)$ is denoted by $\mathcal{MP}_{\gamma_{s, s_1,\dots, s_m}}^{(a)}(^{n_1}E_1,\dots, ^{n_m}E_m; F)$. If $a$ is the origin, we will only write  $\mathcal{MP}_{\gamma_{s, s_1,\dots, s_m}}(^{n_1}E_1,\dots, ^{n_m}E_m; F)$. When $E_1=\cdots = E_m = E$, we will write $\mathcal{MP}_{\gamma_{s, s_1,\dots, s_m}}^{(a)}(^{n_1}E,\dots, ^{n_m}E; F)$. Finally, when it is absolutely  $\gamma_{s, s_1,\dots, s_m}$-summing everywhere, we will write $\mathcal{MP}_{\gamma_{s, s_1,\dots, s_m}}^{(ev)}(^{n_1}E_1,\dots, ^{n_m}E_n; F)$.

Note that
\begin{equation*}
\mathcal{MP}_{\gamma_{s, s_1,\dots, s_m}}^{(ev)}(^{n_1}E_1,\dots, ^{n_m}E_m; F) = \bigcap_{a \in E_1\times \cdots \times E_n} \mathcal{MP}_{\gamma_{s, s_1,\dots, s_m}}^{(a)}(^{n_1}E_1,\dots, ^{n_m}E_m; F).
\end{equation*}


The proof of the following lemma is an immediate consequence of the definition of absolutely  $\gamma_{s, s_1,\dots, s_m}$-summing homogeneous polynomials at the origin.
\begin{lemma}\label{LLL6.2.}
Let $E_1,\dots, E_n$ and $F$ be Banach spaces. So, $P \in \mathcal{MP}_{\gamma_{s, s_1,\dots, s_m}}(^{n_1}E_1,\dots, ^{n_m}E_n; F)$ if, and only if, the induced application
\begin{equation*}
\tilde{P} : \gamma_{s_1}(E_1) \times \cdots \times \gamma_{s_m}(E_m) \rightarrow \gamma_s(F)
\end{equation*}
given by $$\tilde{P}\left(\left(x_j^{(1)} \right)_{j=1}^{\infty},\dots, \left(x_j^{(m)} \right)_{j=1}^{\infty} \right) = \left(P\left(x_j^{(1)},\dots, x_j^{(m)} \right) \right)_{j=1}^{\infty},$$ is well defined. In addition, $\tilde{P}$ is a continuous application.
\end{lemma}

\begin{proof}
It follows immediately from Definition \ref{ashmp} that $\tilde{P}$ is well-defined. The continuity follows from \cite[Proposition $2.4$]{RS17} and \cite[Remark $3.5$]{V17}.

\end{proof}

\begin{proposition}\label{IFF}
Let $E_1,\dots, E_m, F$ be Banach space. So,
\begin{equation*}
P \in \mathcal{MP}_{\gamma_{s, s_1,\dots, s_m}}^{(ev)}(^{n_1}E_1,\dots, ^{n_m}E_m; F) \iff \check{P} \in \mathcal{\prod}_{\gamma_{s, s_1,\overset{n_1}{\dots}, s_1,\dots, s_m,\overset{n_m}{\dots}, s_m}}^{ev}(^{n_1}E_1,\dots, ^{n_m}E_m; F).
\end{equation*}
\end{proposition}

\begin{proof}
Suppose that $\check{P} \in \mathcal{\prod}_{\gamma_{s, s_1,\overset{n_1}{\dots}, s_1,\dots, s_m,\overset{n_m}{\dots}, s_m}}^{ev}(^{n_1}E_1,\dots, ^{n_m}E_m; F)$. 
Then,
\begin{align*}
&\left(P(a_1 + x_j^{(1)},\dots, a_m + x_j^{(m)}) - P(a_1,\dots, a_m) \right)_{j=1}^{\infty}\\
&= \left(\begin{array}{cc}
\check{P}\left(a_1 + x_j^{(1)},\overset{n_1}{\dots}, a_{1} + x_j^{(1)},\dots, a_{m} + x_{j}^{(m)},\overset{n_m}{\dots}, a_m + x_j^{(m)} \right)\\
- \check{P}\left(a_1,\overset{n_1}{\dots},a_{1},\dots, a_{m},\overset{n_m}{\dots}, a_m\right)
\end{array}  \right)_{j=1}^{\infty} \in \gamma_s(F).
\end{align*}
for every $a_1 \in E_1,\dots, a_m \in E_m$ and $\left(x_j^{(1)} \right)_{j=1}^{\infty} \in \gamma_{s_1}(E_1),\dots, \left(x_j^{(m)} \right)_{j=1}^{\infty} \in \gamma_{s_m}(E_m)$. So, $P \in \mathcal{MP}_{\gamma_{s, s_1,\dots, s_m}}^{(ev)}(^{n_1}E_1,\dots, ^{n_m}E_m; F)$.

The other implication is immediate from the $(n_1,\dots, n_m)$-Polarization Formula.

\end{proof}
The next result is immediate of Theorem \ref{CMP} and \cite[Lemma $1$]{S13}.
\begin{proposition}\label{LIM}
If $P \in \mathcal{MP}_{\gamma_{s, s_1,\dots, s_n}}^{(ev)}\left(^{n_1}E_1,\dots, ^{n_m}E_m; F \right)$ and $a = (a_1,\dots, a_m) \in E_1\times \cdots \times E_m$. Then there is a constant $C_{a_1,\dots, a_n} > 0$, such that
\begin{equation*}
\left\|\left(P\left(a_1 + x_j^{(1)},\dots, a_m + x_j^{(m)} \right) - P(a_1,\dots, a_m) \right)_{j=1}^{\infty} \right\|_{\gamma_s(F)} \le C_{a_1,\dots, a_n}
\end{equation*}
whenever $\left(x_j^{(i)} \right)_{j=1}^{\infty} \in B_{\gamma_{s_i}(E_i)}$, $i=1,\dots n$.
\end{proposition}

The argument used in the prove of next proposition is an adaptation of an origin argument due to M. C. Matos in \cite{M03}.
For the next results, consider $G_i = E_i \times \gamma_{s_i}(E_i)$. 


\begin{proposition}
The following statements are equivalent:
\begin{description}
\item $(a)$ $P \in \mathcal{MP}_{\gamma_{s, s_1,\dots, s_n}}^{m, ev}(^{n_1}E_1,\dots, ^{n_m}E_n; F)$;

\item $(b)$ The induced operator
\begin{equation*}
\Phi(P) : G_1 \times \cdots \times G_m \rightarrow \gamma_s(F)
\end{equation*}
given by
\begin{align*}
&\Phi(P)\left(\left(a_1, \left(x_j^{(1)} \right)_{j=1}^{\infty} \right),\dots, \left( a_m, \left(x_j^{(m)} \right)_{j=1}^{\infty} \right) \right)\\
&= \left(P\left(a_1 + x_j^{(1)},\dots, a_m + x_j^{(m)} \right) - P(a_1,\dots, a_m) \right)_{j=1}^{\infty}.
\end{align*}
is well-defined $(n_1,\dots, n_m)$-homogeneous polynomials. In addition, $\Phi(T)$ is a continuous application.

\item $(c)$ There exists $C > 0$, such that,
\begin{align}\label{DESI}
&\left\|\left(P\left(a_1+x_j^{(1)},\dots, a_m+x_j^{(m)} \right) - P(a_1,\dots, a_m) \right)_{j=1}^{\infty} \right\|_{\gamma_s(F)}\nonumber\\
&\le C \prod_{i=1}^{m}\left(\|a_i\|_{E_i} + \left\|\left(x_j^{(i)} \right)_{j=1}^{\infty} \right\|_{\gamma_{s_i}(F)} \right)^{n_i},
\end{align}
whenever $\left(x_j^{(i)} \right)_{j=1}^{\infty} \in \gamma_{s_i}(E_i)$, $i=1,\dots, m$ and $(a_1,\dots, a_m) \in E_1 \times \cdots \times E_m$.

\item $(d)$ There exists $C > 0$, such that,
\begin{align*}
&\left\|\left(P\left(a_1+x_j^{(1)},\dots, a_m+x_j^{(m)} \right) - P(a_1,\dots, a_m) \right)_{j=1}^{n} \right\|_{\gamma_s(F)}\\
&\le C \prod_{i=1}^{m}\left(\|a_i\|_{E_i} + \left\|\left(x_j^{(i)} \right)_{j=1}^{n} \right\|_{\gamma_{s_i}(F)} \right)^{n_i},
\end{align*}
whenever $\left(x_j^{(i)} \right)_{j=1}^{\infty} \in \gamma_{s_i}(E_i)$, $i=1,\dots, m$ and $(a_1,\dots, a_m) \in E_1 \times \cdots \times E_m$.
\end{description}
\end{proposition}

By straightforward computations, we can get the following result.

\begin{corollary}\label{NORM}
\begin{description}
\item $(a)$ The infimum of the constants $C > 0$ satisfying \eqref{DESI} defines a norm in $\mathcal{MP}_{\gamma_{s, s_1,\dots, s_m}}^{(ev)}(^{n_1}E_1,\dots, ^{n_m}E_m; F)$, that is denoted by $\pi_{\gamma_{s, s_1,\dots, s_m}}^{(ev)}(\cdot)$.
\item $(b)$ If $P \in \mathcal{MP}_{\gamma_{s, s_1,\dots, s_m}}^{(ev)}(^{n_1}E_1,\dots, ^{n_m}E_m; F)$, then $\pi_{\gamma_{s, s_1,\dots, s_m}}^{(ev)}(P) = \|\Phi(P)\|$.
\item $(c)$ If $P \in \mathcal{MP}_{\gamma_{s, s_1,\dots, s_m}}^{(ev)}(^{n_1}E_1,\dots, ^{n_m}E_m; F)$, then $\|P\| \le \pi_{\gamma_{s, s_1,\dots, s_m}}^{(ev)}(P)$.
\end{description}
\end{corollary}



\begin{proposition}
The norm $\pi_{\gamma_{s, s_1,\dots, s_m}}^{(ev)}(\cdot)$ defined in Corollary \ref{NORM}, satisfies the relation
\begin{equation*}
\pi_{\gamma_{s, s_1,\dots, s_m}}^{(ev)}(P) \le \pi_{\gamma_{s, s_1,\overset{n_1}{\dots},s_1,\dots, s_m,\overset{n_m}{\dots} s_m}}^{ev}(\check{P}) \le \frac{n_1^{n_1}\cdots n_m^{n_m}}{n_1!\cdots n_m!}\pi_{\gamma_{s, s_1,\dots, s_m}}^{(ev)}(P).
\end{equation*}
\end{proposition}

Following standard arguments, it is possible to prove that:

\begin{theorem}
Suppose that $\gamma_{s_1}(\mathbb{K})\overset{n_1}{\cdots}\gamma_{s_1}(\mathbb{K})\cdots \gamma_{s_m}(\mathbb{K})\overset{n_m}{\cdots}\gamma_{s_m}(\mathbb{K}) \overset{1}{\hookrightarrow} \gamma_s(\mathbb{K})$, then
\begin{equation*}
\left(\mathcal{MP}_{\gamma_{s, s_1,\dots, s_m}}^{(ev)}(^{n_1}(\cdot),\dots, ^{n_m}(\cdot); (\cdot)), \pi_{\gamma_{s, s_1,\dots, s_m}}^{(ev)}(\cdot) \right)
\end{equation*}
is a Banach ideal.
\end{theorem}

\begin{definition}
Let $\mathcal{M}_{n_1+\cdots +n_m}$ be an ideal of $(n_1+\dots+ n_m)$-linear applications. We say that $\mathcal{M}_{n_1+\cdots +n_m}$ is $(n_1,\dots, n_m)$-symmetric if $P_{s}^{(n_1,\dots, n_m)} \in \mathcal{M}_{n_1+\cdots +n_m}$ whenever $P \in \mathcal{M}_{n_1+\cdots +n_m}$.
\end{definition}

\begin{theorem}
The $\mathcal{\prod}_{\gamma_{s, s_1,\overset{n_1}{\dots}, s_1,\dots, s_m,\overset{n_m}{\dots}, s_m}}^{ev}$ is an $(n_1,\dots, n_m)$-symmetric ideal.
\end{theorem}

\begin{proof}
Let $E_1,\dots, E_m, F$ be Banach spaces and $T \in \mathcal{\prod}_{\gamma_{s, s_1,\overset{n_1}{\dots}, s_1,\dots, s_m,\overset{n_m}{\dots}, s_m}}^{ev}(^{n_1}E_1,\dots, ^{n_m}E_m; F)$. Then, for every $a_1^{(i)},\dots, a_{n_i}^{(i)} \in E_i$,
$\left(x_j^{(i, 1)} \right)_{j=1}^{\infty},\dots, \left(x_j^{(i, n_i)} \right)_{j=1}^{\infty} \gamma_{s_i}(E_i)$
and $\sigma_i \in S_{n_i}$, $i=1,\dots, m$
\begin{align*}
\left(\begin{array}{cc}
T\left(a_{\sigma_1(1)}^{(1)} + x_j^{(1, \sigma_1(1))},\dots, a_{\sigma_1(n_1)}^{(1)} + x_j^{(1, \sigma_1(n_1))},\dots, a_{\sigma_m(1)}^{(m)} + x_j^{(m, {\sigma_m(1)})},\dots, a_{{\sigma_m(n_m)}}^{(m)} + x_j^{(m, {\sigma_m(n_m)})} \right)\\
- T\left(a_{\sigma_1(1)}^{(1)},\dots, a_{\sigma_1(n_1)}^{(1)},\dots, a_{\sigma_m(1)}^{(m)},\dots, a_{{\sigma_m(n_m)}}^{(m)} \right)
\end{array}  \right)_{j=1}^{\infty}
\end{align*}
belongs to $\gamma_s(F)$. Since $\mathcal{\prod}_{\gamma_{s, s_1,\overset{n_1}{\dots}, s_1,\dots, s_m,\overset{n_m}{\dots}, s_m}}^{ev}(^{n_1}E_1,\dots, ^{n_m}E_m; F)$ is a linear space, it follows that $T_s^{(n_1,\dots, n_m)} \in \mathcal{\prod}_{\gamma_{s, s_1,\overset{n_1}{\dots}, s_1,\dots, s_m,\overset{n_m}{\dots}, s_m}}^{ev}(^{n_1}E_1,\dots, ^{n_m}E_m; F)$.
\end{proof}

\end{document}